\pdfoutput=1                      

\documentclass[11pt,reqno]{amsart}

\usepackage[accsupp]{axessibility}    

\usepackage[a4paper,margin=1.2in]{geometry}

\usepackage{amsmath}
\usepackage{amssymb}
\usepackage{amsthm}
\usepackage{color}
\usepackage{verbatim}
\usepackage{tabularx}

\usepackage{mathrsfs}
\usepackage{hyperref}

\theoremstyle{theorem}
\newtheorem{theorem}{Theorem}
\newtheorem*{theorem*}{Theorem}
\newtheorem{introtheorem}{Theorem}

\newtheorem{conjecture}[theorem]{Conjecture}
\newtheorem{corollary}[theorem]{Corollary}

\newtheorem{lemma}[theorem]{Lemma}

\theoremstyle{definition}
\newtheorem{remark}[theorem]{Remark}

\newcommand{\R}{\mathbb{R}}

\newcommand{\eps}{\varepsilon}

\newcommand{\II}{\mathrm{II}}

\newcommand{\di}{\mathrm{d}}

\newcommand{\HH}{\mathbb{H}}

\newcommand{\Ricc}{\mathrm{Ric}}

\newcommand{\vol}{\mathrm{vol}}
\newcommand{\area}{\mathrm{Area}}
\newcommand{\lip}{\mathrm{Lip}}

\newcommand{\Sec}{\mathrm{Sec}}

\newcommand{\ber}{\mathscr{B}}

\DeclareMathOperator{\dist}{dist}

\renewcommand{\div}{\mathrm{div}}

\title[Minimal graphs with sublinear growth]{On minimal graphs of sublinear growth over manifolds with non-negative Ricci curvature}

\author{Giulio Colombo}
\address{Dipartimento di Matematica e Applicazioni ``R. Caccioppoli", Universit\`a degli Studi di Napoli ``Federico II'', Via Vicinale Cupa Cintia 26, I-80126 Napoli (Italy).}
\email{giulio.colombo@unina.it}

\author{Luciano Mari}
\address{Dipartimento di Matematica ``G. Peano", Universit\`a degli Studi di Torino, Via Carlo Alberto 10, 10123 Torino (Italy).}
\email{luciano.mari@unito.it}

\author{Marco Rigoli}
\address{Dipartimento di Matematica ``F. Enriques", Universit\`a degli Studi di Milano, Via Saldini 50, I-20133 Milano (Italy).}
\email{marco.rigoli55@gmail.com}

\begin{document}

\maketitle

\begin{abstract}
We prove that entire solutions of the minimal hypersurface equation 
	\begin{equation}
		\div\left(\frac{Du}{\sqrt{1+|Du|^2}}\right) = 0
	\end{equation}
on a complete manifold with $\Ricc \ge 0$, whose negative part grows like $\mathcal{O}(r/\log r)$ ($r$ the distance from a fixed origin), are constant. This extends the Bernstein Theorem for entire positive minimal graphs established in recent years. The proof depends on a new technique to get gradient bounds by means of integral estimates, which does not require any further geometric assumption on $M$. 
\end{abstract}

\bigskip



\bigskip

The aim of this note is to prove the following Liouville-type theorem for entire solutions $u$ of the minimal hypersurface equation on complete manifolds of non-negative Ricci curvature. Hereafter, we denote with
\[
u_-(x) = \max\{0,-u(x)\}
\]
the negative part of $u$. 

\begin{introtheorem} \label{thm_main2}
	Let $M$ be a connected, complete Riemannian manifold with $\Ricc\geq 0$ and $u\in C^\infty(M)$ a solution of the minimal hypersurface equation
	\begin{equation}
		\div\left(\frac{Du}{\sqrt{1+|Du|^2}}\right) = 0 \qquad \text{on } \, M.
	\end{equation}
	Let $r$ be the Riemannian distance from a fixed point $o\in M$. If
	\begin{equation} \label{uArg}
	u_-(x) = \mathcal{O}\left(\frac{r(x)}{\log r(x)}\right) \qquad \text{as } \, r(x) \to \infty,
	\end{equation}
	then $u$ is constant.
\end{introtheorem}

The interest in this problem stems from a series of classical and more recent results concerning the global behaviour of entire solutions of the minimal hypersurface equation both on Euclidean spaces and on suitable classes of complete Riemannian manifolds. In particular, the result improves on the Bernstein Theorem for positive minimal graphs on complete manifolds with $\Ricc \ge 0$ recently proved by the authors with Magliaro \cite{cmmr} and by Ding \cite{ding21}. The main novelty of the present paper is that we can treat, without any additional requirement besides $\Ricc \ge 0$, solutions whose negative part diverges in a controlled way, close to the sharp condition $u_-(x) = o(r(x))$ under which it is expected that the conclusion of Theorem \ref{thm_main2} may hold as well in the stated assumptions.

\section{Introduction}

In Euclidean space $\R^{m+1}$, the (generalized) Bernstein theorem considers minimal hypersurfaces $\Sigma$ given as the graph of a function $u : \R^m \to \R$, which therefore solves the minimal hypersurface equation
	\begin{equation} \tag{MSE} \label{MSE}
		\div\left(\frac{Du}{\sqrt{1+|Du|^2}}\right) = 0. 
	\end{equation}
Having defined the following property:
\begin{itemize}
\item[$(\ber 1)$] all solutions to \eqref{MSE} are affine functions,
\end{itemize}
the theorem guarantees that 
\[
\begin{array}{c}
\text{$(\ber 1)$ holds for solutions}\\
\text{$u : \R^m \to \R$ to \eqref{MSE}}
\end{array} \qquad \Longleftrightarrow \qquad m \le 7.
\]
For a detailed treatment of this cornerstone result, see \cite{giusti}. On the other hand, if one assumes an a-priori bound on $u$ (or $|Du|$), then Bernstein-type theorems holding true in every dimension $m \ge 2$ have been obtained by Moser \cite{mos61} and Bombieri, De Giorgi and Miranda \cite{bdgm}. In particular, the main result in \cite{bdgm} guarantees the following properties:
\begin{itemize}
\item[$(\ber 2)$] positive solutions to \eqref{MSE} are constant;
\item[$(\ber 3)$] solutions to \eqref{MSE} with at most linear growth on one side, namely, satisfying
\begin{equation}\label{eq_linear_oneside}
u_-(x) = \mathcal{O}\big(r(x)\big) \qquad \text{as } \, r(x) \to \infty
\end{equation}
are affine functions. 
\end{itemize}
Here, $r(x)$ is the distance from a fixed origin. Notice that $(\ber 3)$ implies $(\ber 2)$, indeed in the following strengthened form:
\begin{itemize}
\item[$(\ber 2')$] Solutions to \eqref{MSE} satisfying $u_-(x) = o\big(r(x)\big)$ as $r(x) \to \infty$ are constant.  
\end{itemize}
Also, $(\ber 3)$ implies Moser's theorem in \cite{mos61}: globally Lipschitz solutions to \eqref{MSE} are affine. 

\vspace{0.3cm}

A natural question one might ask is under which conditions a similar picture occurs for minimal graphs in more general ambient spaces. Topological products $M \times \R$ with $(M^m,\sigma)$ a complete connected Riemannian manifold are a natural setting. Given $u : M \to \R$ and the associated graph map 
\[
\Gamma_u \ : \ M \to M \times \R, \qquad \Gamma_u(x) = (x,u(x)),
\]
the possible rigidity of $\Sigma = \Gamma_u(M)$ depends on the metric chosen on $M \times \R$, and here we focus on the product metric $\sigma + \di t^2$. Then, identifying $\Sigma$ with $M$ endowed with the induced metric $g = \Gamma_u^*(\sigma + \di t^2)$, $\Sigma$ is minimal if and only if $u$ solves \eqref{MSE} on $M$, where now $D$, $\div$ are the gradient and divergence in $(M,\sigma)$. 

\begin{remark}
The problem has also been considered for different warped product metrics on $M \times \R$, we refer to \cite{bcmmpr,bmpr} for motivations and a detailed account.  
\end{remark}

By \cite{nr02,dfr10} (see also \cite{dvh94,dvh95} for previous achievements), in hyperbolic space $\HH^m$ any continuous boundary value $\phi \in C(\partial_\infty \HH^m)$ gives rise to a bounded solution to \eqref{MSE} attaining $\phi$ at infinity. Hence, there are plenty of bounded minimal graphs over $\HH^m$, and $(\ber 1),(\ber 2),(\ber 3)$ fail. The results have been extended to Cartan-Hadamard manifolds with suitably pinched negative sectional curvature, see \cite{bcmmpr} and the references therein. On the other hand, for reasons explained in \cite{bcmmpr,cgmr,cmmr}, results analogous to those for $M=\R^m$ might be expected on complete manifolds $(M,\sigma)$ with non-negative sectional or Ricci curvature: 
\[
\Sec \ge 0, \qquad \text{or} \qquad \Ricc \ge 0.
\]
The analogy is also suggested by the theory of harmonic functions, in particular by the classical results of Yau and Cheng-Yau \cite{yau, chengyau}, Kasue \cite{kasue} and Cheeger-Colding-Minicozzi \cite{ccm95}. In this respect, notice that on the graph $\Sigma = (M,g)$ equation \eqref{MSE} rewrites as $\Delta_g u = 0$. 

Apart from the Euclidean space, $(\ber 1)$ was proved on manifolds with $\Ricc \ge 0$ and satisfying the mild volume growth  condition 

\begin{equation}\label{eq_mildvol}
\int^\infty \frac{r \di r}{\vol(B_r)} = \infty.
\end{equation}
Indeed, one gets the following rigidity result, well-known at least if $M$ has dimension $2$ (a case for which \eqref{eq_mildvol} is automatically satisfied if $\Ricc \ge 0$). See \cite[Theorem 6(i) and Remark 5]{cgmr}.

\begin{theorem}[\cite{cgmr}]\label{teo_B1}
Let $(M,\sigma)$ be a complete manifold with $\Ricc \ge 0$ and satisfying \eqref{eq_mildvol}, and let $u : M \to \R$ be a non-constant solution to \eqref{MSE}. Then, $M$ admits a splitting $N \times \R$ with the product metric $\sigma_N + \di s^2$ such that $u(y,s) = as + b$ for some $a,b \in \R$.
\end{theorem}

To the best of our knowledge, property $(\ber 1)$ seems far from reaching on manifolds not satisfying \eqref{eq_mildvol}, even in case $\Sec \ge 0$. On the other hand, much progress was made in recent years for properties $(\ber 2),(\ber 3)$. When $\Sec \ge 0$, by \cite[Corollary 10]{cgmr} property $(\ber 3)$ holds: \begin{quote}
If $M$ is complete, connected with $\Sec \ge 0$, and if $u$ is a non-constant solution to \eqref{MSE} satisfying \eqref{eq_linear_oneside}, then the conclusion of Theorem \ref{teo_B1} holds. 
\end{quote}  
Consequently, $(\ber 2)$ and $(\ber 2')$ hold as well.\par
Compared to the case of non-negative sectional curvature, the problem on manifolds only satisfying $\Ricc \ge 0$ presents new challenges. To justify this statement, we first observe that the techniques in \cite{yau, chengyau, ccm95} to study harmonic functions rely in an essential way on two ingredients:
\begin{itemize}
\item[(i)] The Bochner formula for solutions to $\Delta u = 0$;
\item[(ii)] The properties of the distance $r$ to a fixed point in $(M, \sigma)$, in particular the Laplacian bound
\begin{equation}\label{eq_deltar}
\Delta r \le \frac{m-1}{r}
\end{equation}
which follows from comparison theory.
\end{itemize}
Trying to follow the same approach when $u$ solves \eqref{MSE} instead of $\Delta u = 0$, a natural starting point in place of the Bochner formula is the Jacobi equation
\[
\Delta_g W^{-1} + \left( \|\II\|^2 + \Ricc\left( \frac{Du}{W},\frac{Du}{W} \right) \right) W^{-1} = 0
\]
where $W=\sqrt{1+|Du|^2}$, motivating once more the interest in studying the behaviour of solutions under the sole assumption $\Ricc \ge 0$. However, (ii) entails to study $\Delta_g r$, which in view of the minimality of $\Sigma$ can be written as 
\[
\Delta_g r = g^{ij}(D^2r)_{ij}
\] 
in local coordinates $\{x^i\}$ on $M$. The operator $\Delta_g$ fails to be uniformly elliptic if $|Du|$ is unbounded. In particular, while the lower bound $\Sec \ge 0$ allows to estimate $\Delta_g r$ by an expression like the one in \eqref{eq_deltar}, and therefore to adapt most of the arguments known for harmonic functions, the bound $\Ricc \ge 0$ alone seems insufficient, and calls for new ideas. On the other hand, the techniques in \cite{bdgm,bg} to get $(\ber 2),(\ber 3)$ in Euclidean space heavily rely on the fact that $\Sigma$ enjoys the isoperimetric inequality
\[
\left(\int_\Sigma \phi^{\frac{m}{m-1}}\right)^{\frac{m-1}{m}} \le \mathscr{S} \int_\Sigma \|\nabla \phi\| \qquad \forall \, \phi \in \lip_c(\Sigma), 
\]
which is known to fail in our setting unless $M$ has maximal volume growth, namely, balls $B_R\subset (M,\sigma)$ centered at a fixed origin satisfy
\[
\lim_{R \to \infty} \frac{\vol(B_R)}{R^m} > 0.
\]
New approaches able to overcome the aforementioned problems are, to our opinion, interesting on their own, and may allow to application to other relevant PDEs. \par

The validity of $(\ber 2)$ on complete manifolds with $\Ricc \ge 0$ was recently shown by the authors in a joint paper with Magliaro \cite{cmmr}, and also independently by Ding \cite{ding21} with different methods. The result improves on previous work of Rosenberg, Schulze and Spruck \cite{rss13} by getting rid of the technical assumption that the sectional curvature of $(M,\sigma)$ be bounded from below. We have:

\begin{theorem*}[\cite{cmmr},\cite{ding21}]
	Let $M$ be a complete, connected Riemannian manifold with $\Ricc\geq0$. Then any non-negative solution to \eqref{MSE} on $M$ is constant.
\end{theorem*}

\noindent The proof in \cite{cmmr} is based on a new global gradient estimate for positive solutions to \eqref{MSE} on complete manifolds with Ricci curvature bounded below by a (possibly negative) constant. The estimate is obtained by using, in place of the distance function $r$, a different exhaustion function constructed by means of potential theory and a duality principle recently established in \cite{marivaltorta, maripessoa}. By contrast, Ding's argument in \cite{ding21} is based on a new Harnack-type inequality for positive solutions to \eqref{MSE}. In fact, he considers complete manifolds satisfying global doubling and weak Neumann-Poincar\'e inequalities, a class including those with $\Ricc \ge 0$. In his paper, Ding also remarks that from the same Harnack inequality, by applying standard methods first outlined in \cite{mos61}, it is possible to infer the existence of $\delta>0$ such that if $u$ is a solution to \eqref{MSE} satisfying 
\begin{equation}\label{eq_odelta}
u_-(x) = o\big(r(x)^\delta\big) \qquad \text{as } \, r(x) \to \infty
\end{equation}
then $u$ is constant. However, he does not provide any estimate on the sharp value of $\delta$. In this respect, our main Theorem \ref{thm_main2} improves on \cite{cmmr,ding21} and in particular shows the constancy of $u$ whenever \eqref{eq_odelta} holds for any $\delta\in(0,1)$. Indeed, Theorem \ref{thm_main2} goes in the direction of proving $(\ber 2')$, failing only by a logarithmic term.

We conclude this introduction by briefly discussing properties $(\ber 2')$ and $(\ber 3)$ on manifolds with $\Ricc \ge 0$. First, $(\ber 3)$ does not hold in the same strong form as for manifolds with $\Sec \ge 0$. Indeed, by \cite[Proposition 9]{cgmr} there exists a complete manifold with $\Ricc > 0$ (hence, not splitting off any line) supporting an entire solution to \eqref{MSE} with bounded gradient. This is not surprising, since by an example in  \cite{kasuewashio} the same happens for solutions to $\Delta u = 0$, and in fact \cite{cgmr} elaborates on \cite{kasuewashio}. Then, the analogy with the case of harmonic functions discussed in \cite{ccm95} led us to formulate the following:
\begin{conjecture}
	If $M$ is complete with $\Ricc \ge 0$, and if $u$ is a non-constant solution to \eqref{MSE} satisfying \eqref{eq_linear_oneside}, then every tangent cone at infinity $M_\infty$ splits isometrically as $M_\infty = N_\infty \times \R$.
\end{conjecture}

In \cite[Theorem 8]{cgmr}, the conjecture was verified under the stronger assumption that $u$ is globally Lipschitz. In fact, in the proof of the theorem it is also shown that any blowdown $u_\infty : M_\infty \to \R$ of $u$ corresponding to a tangent cone $M_\infty = N_\infty \times \R$ satisfies $u_\infty(y,t) = \|Du\|_\infty t$, from which one can deduce that the original function $u$ has exactly linear growth on both sides. In particular, under the additional assumption $|Du|\in L^\infty(M)$, $(\ber 2')$ holds as well; indeed, the latter fact also follows from a more direct argument, sketched in the proof of \cite[Theorem 11]{cgmr}, which we describe in detail in Lemma \ref{lem_liou} below.

In view of \cite[Theorem 8]{cgmr}, if we could prove that
\begin{equation} \label{B2Du}
	\text{$u$ satisfies \eqref{eq_linear_oneside}} \qquad \Longrightarrow \qquad |Du| \in L^\infty(M),
\end{equation}
then the conjecture would be verified, and due to the above observation $(\ber 2')$ would follow too. Currently, the implication \eqref{B2Du} was only shown under further geometric assumptions, precisely:
\begin{itemize}
\item[$(i)$] in \cite[Corollary 18]{cgmr}, \eqref{B2Du} was proved in case the $(m-2)$-th Ricci curvature of $(M, \sigma)$ in radial direction from a fixed origin has a lower bound decaying quadratically to zero;
\item[$(ii)$] in \cite{ding_new}, \eqref{B2Du} was proved in case $(M,\sigma)$ has maximal volume growth, as a consequence of Theorem 1.3 there.
\end{itemize}
Note that the two extra assumptions in $(i),(ii)$ are unrelated. Compared to \cite{cgmr,ding_new}, the core of the proof of Theorem \ref{thm_main2} is to obtain a global gradient estimate for $u$ under the validity of \eqref{uArg}, without any extra condition on $M$. The technique is based on integral estimates that, differently from \cite{bdgm,bg,ding_new}, do not require the validity of an isoperimetric inequality on $\Sigma$ and, unlike \cite{tru72,korevaar,rss13,cgmr}, do not use second order properties of the distance function $r$. It might be possible that our method be refined to get global gradient bounds under \eqref{eq_linear_oneside}.

As a partial result in the attempt to prove $(\ber 2')$, in the last part of the paper we obtain an asymptotic upper estimate of the measure of superlevel sets of the function $|Du|$ in case $u$ is an entire solution to \eqref{MSE} satisfying a two-sided bound $|u|=o(r)$. We show that
\[
	\lim_{R\to \infty} \frac{\vol(\{|Du|>\eps\}\cap B_R)}{\vol(B_R)} = 0 \qquad \forall \, \eps>0 \, .
\]
We refer to Corollary \ref{cor_q1} for a precise statement. We also point out that, as a consequence of Ding's Harnack-type inequality for positive solutions of \eqref{MSE} proved in \cite{ding21}, the bound $|u|=o(r)$ is implied by the weaker, one-sided condition $u_- = o(r)$.

\section{Preliminary lemmas}

Let $(M,\sigma)$ be a connected, complete Riemannian manifold of dimension $m\geq2$. We denote by $|\cdot|$, $D$, $\div$, $\vol$ and $\di v\doteq\di\vol$ the vector norm, gradient operator, divergence operator, volume measure and volume form induced by $\sigma$, respectively. For $u : M \to \R$ a smooth function, we denote by
\[
	\Gamma_u : M \to M \times \R, \qquad x \mapsto (x,u(x))
\]
the graph map, and by $\Sigma = \Gamma_u(M)$ the graph of $u$. Having chosen the ambient product metric $\bar\sigma = \sigma + \di t^2$, we let $g = \Gamma_u^* \bar \sigma$ the induced graph metric on $\Sigma$, which is therefore isometric to $(M,g)$. We denote by $\|\cdot\|$, $\nabla$, $\div_g$, $\Delta_g$, $\vol_g$ and $\di v_g\doteq\di\vol_g$ the vector norm, gradient, divergence, Laplace-Beltrami operator, volume measure and volume form induced by $g$, respectively. The metric $g$ and its volume form $\di v_g$ satisfy
\[
	g = \sigma + \di u \otimes \di u \, , \qquad \di v_g = W \di v
\]
where $W\in C^\infty(M)$ is defined, as usual, by
\begin{equation} \label{Wdef}
	W = \sqrt{1+|Du|^2} \, .
\end{equation}

With respect to any given local coordinate system $\{x^i\}$ on $M$ we write
\[
	\sigma = \sigma_{ij} \, \di x^i \otimes \di x^j \, , \qquad g = g_{ij} \, \di x^i \otimes \di x^j
\]
where the Einstein convention of summation over repeated indices is in force. Letting $\sigma^{ij}$ and $g^{ij}$ denote the coefficients of the inverse matrices $(\sigma_{ij})^{-1}$ and $(g_{ij})^{-1}$, respectively, we have
\begin{equation} \label{gij}
	g^{ij} = \sigma^{ij} - \frac{u^i u^j}{W^2}
\end{equation}
where $u^i = \sigma^{ij} \frac{\partial u}{\partial x^j}$. Using \eqref{gij} and \eqref{Wdef} we compute the norm of the gradient of $u$ in the graph metric $g$ to be
\begin{equation} \label{nablau}
	\|\nabla u\|^2 = g^{ij} \frac{\partial u}{\partial x^i} \frac{\partial u}{\partial x^j} = \frac{1}{W^2} \sigma^{ij} \frac{\partial u}{\partial x^i} \frac{\partial u}{\partial x^j} = \frac{|Du|^2}{W^2} \equiv 1 - \frac{1}{W^2} \, .
\end{equation}
By \eqref{gij} and \eqref{Wdef} we also have, for any $\psi\in C^1(M)$, the following useful formulas involving the gradients $D\psi$ and $\nabla\psi$ computed with respect to the metrics $\sigma$ and $g$, respectively:
\begin{align}
	\label{nabla_pu}
	\langle\nabla\psi,\nabla u\rangle & = g^{ij} \frac{\partial u}{\partial x^i} \frac{\partial\psi}{\partial x^j} = \left(1-\frac{|Du|^2}{W^2}\right) \sigma^{ij} \frac{\partial u}{\partial x^i} \frac{\partial\psi}{\partial x^j} = \frac{\sigma(Du,D\psi)}{W^2} \, , \\
	\|\nabla\psi\|^2 & = |D\psi|^2 - \frac{\sigma(D\psi,Du)^2}{W^2} \, .
\end{align}
From the second one we have the two-sided estimate
\begin{equation}
	\frac{|D\psi|^2}{W^2} \leq \|\nabla\psi\|^2 \leq |D\psi|^2 \, .
\end{equation}
Lastly, for any $\psi\in C^2(M)$ we have the following well-known expression for the graph Laplacian $\Delta_g$
\begin{equation}
	\Delta_g \psi = \frac{1}{\sqrt{|g|}} \sum_{i,j=1}^m \frac{\partial}{\partial x^i} \left( \sqrt{|g|} g^{ij}\frac{\partial\psi}{\partial x^j}\right) \equiv \frac{1}{W\sqrt{|\sigma|}} \sum_{i,j=1}^m \frac{\partial}{\partial x^i} \left( W \sqrt{|\sigma|} g^{ij}\frac{\partial\psi}{\partial x^j}\right)
\end{equation}
where $|g| = \det(g_{ij}) = W^2\det(\sigma_{ij}) = W^2|\sigma|$. In other words,
\[
	\Delta_g \psi = \frac{1}{W} \div(\mathbf{A}D\psi)
\]
where $\mathbf{A}$ is the endomorphism of $TM$ of components $\mathbf{A}^i_k = W g^{ij}\sigma_{jk}$. We remark that the eigenvalues of $\mathbf{A}$ are $W$ and $W^{-1}$. Hence, whenever $|Du|\in L^\infty(M)$ we have that $W\Delta_g$ is a uniformly elliptic second order operator in divergence form on $(M,\sigma)$. This will be crucial in the last part of the proof of the main theorem.

From now on, we always assume that $u$ is a solution of the minimal hypersurface equation \eqref{MSE}, which is equivalent to $\Delta_g u = 0$ on $\Sigma$. Since $\partial_t$ is a Killing field, the positive function $1/W$ solves the Jacobi equation
\begin{equation} \label{Jac}
	\Delta_g \frac{1}{W} = - (\|\II\|^2 + \Ricc_{\bar\sigma}(\mathbf{n},\mathbf{n})) \frac{1}{W} \qquad \text{on } \, \Sigma
\end{equation}
where $\II$ is the second fundamental form of $\Sigma$ and $\mathbf{n}$ is a normal vector field to $\Sigma$ in $M\times\R$. As a consequence, for any $C>0$ the function
\begin{equation} \label{z_def}
	z := W e^{-Cu}
\end{equation}
satisfies (see \cite[formula (33) with $H=0$]{cmr21})
\[
	\div_g(W^{-2}\nabla z) = \left(\|\II\|^2 + \Ricc_{\bar\sigma}(\mathbf{n},\mathbf{n}) + C^2\|\nabla u\|^2 \right) W^{-2} z \qquad \text{on } \, \Sigma \, .
\]
In case $(M,\sigma)$ has $\Ricc\geq 0$ this yields
\begin{equation} \label{divWz}
	\div_g(W^{-2}\nabla z) \geq C^2\|\nabla u\|^2 W^{-2} z \qquad \text{on } \, \Sigma \, .
\end{equation}
The idea of considering an auxiliary function of the form \eqref{z_def} dates back to Korevaar, \cite{korevaar}.

We now fix an origin $o\in M$ and set $r(x) = \dist_\sigma(o,x)$ for each $x\in M$. For any $R>0$ we denote
\[
	B_R = B_R^\sigma(o) = \{ x\in M : r(x) < R \}
\]
the geodesic ball in $(M,\sigma)$ of radius $R$ centered at $o$, and by
\begin{equation}
	\Sigma_R = \{ (x,u(x))  : x\in B_R \}
\end{equation}
the intersection of $\Sigma$ with the infinite cylinder $B_R\times\R\subseteq M\times\R$. For all $-\infty<a<b<\infty $ we also set
\begin{equation} \label{SRab}
	\Sigma_{R,a,b} = \{ (x,u(x)) : x \in B_R , \, a<u(x)<b \} \, .
\end{equation}

We conclude this section by collecting three lemmas needed in the proof of the main result. The first two lemmas give upper estimates on the measure $\vol_g(\Sigma_{R,a,b})$ and on the integral of $e^{-Cu}\di v_g$ on $\Sigma_R$ in case $(M,\sigma)$ is any complete Riemannian manifold, in the spirit, respectively, of \cite{tru72} (see also \cite[Theorem 3]{mir67}) and \cite{bdgm}. The third lemma provides a weighted Caccioppoli-type inequality for $u$ on superlevel sets of the auxiliary function $z$ in the assumption that $(M,\sigma)$ is also of non-negative Ricci curvature.

We briefly comment on the geometrical meaning of the upper estimates \eqref{Vest0}-\eqref{Vest2} for $\vol_g(\Sigma_{R,a,b})$ given by the first lemma below. For any $R>0$ and $-\infty<a<b<\infty $, the set $\Sigma_{R,a,b}$ defined in \eqref{SRab} is the intersection of $\Sigma$ with the cylinder
\[
	\Gamma_{R,a,b} = B_R \times (a,b)
\]
of $M\times\R$. By a classical calibration argument (see, for instance, the first part of the proof of Lemma 3.2 in \cite{cmmr}), the measure $\vol_g(\Sigma_{R,a,b}) \equiv \vol_g(\Sigma\cap\Gamma_{R,a,b})$ is not larger than half of the perimeter of the surrounding cylinder $\Gamma_{R,a,b}$, so for a.e.~$R>0$ we have
\[
	\vol_g(\Sigma_{R,a,b}) \leq \vol(B_R) + \frac{b-a}{2} \, \area(\partial B_R) \, .
\]
A quite standard variation of this argument yields an analogous estimate in which $\area(\partial B_R)$ is replaced with a term depending on the volume $\vol(B_{R_1})$ of a second ball of larger radius $R_1>R$. This turns out to be useful in case only a control on the growth of the function $R\mapsto\vol(B_R)$ is available.

\begin{lemma}
	Let $(M,\sigma)$ be a complete Riemannian manifold, $u\in C^\infty(M)$ a solution to \eqref{MSE} and $o\in M$ a fixed origin. For any $-\infty<a<b<\infty $ and $R_1>R>0$ we have
	\begin{equation} \label{Vest0}
		\vol_g(\Sigma_{R,a,b}) \leq \vol(B_R) + \frac{b-a}{2} \frac{\vol(B_{R_1}\setminus B_R)}{R_1-R}
	\end{equation}
	where $\Sigma_{R,a,b}$ is as in \eqref{SRab}. In particular, for any $-\infty<a<b<\infty $ we have
	\begin{equation} \label{Vest1}
		\vol_g(\Sigma_{R,a,b}) \leq \vol\left(B_{R+\frac{1}{2}(b-a)}\right) \qquad \forall \, R>0
	\end{equation}
	and also
	\begin{equation} \label{Vest2}
		\vol_g(\Sigma_{R,a,b}) \leq \vol(B_R) + \frac{b-a}{2} \, \area(\partial B_R) \qquad \text{for a.e.~} R>0 \, .
	\end{equation}
\end{lemma}

\begin{proof}
	Let $\psi : M\to[0,1]$ be the Lipschitz cut-off function defined by
	\[
	\psi = \begin{cases}
		1 & \quad \text{on } \, B_R \\
		\dfrac{R_1-r}{R_1-R} & \quad \text{on } \, B_{R_1}\setminus B_R \\
		0 & \quad \text{on } \, M\setminus B_{R_1}
	\end{cases}
	\]
	and let $\hat u = \chi\circ u$, where $\chi : \R \to \R$ is given by
	\[
	\chi(s) = \begin{cases}
		\frac{1}{2}(a-b) & \quad \text{if } s \leq a \\
		s-\frac{1}{2}(a+b) & \quad \text{if } \, a<s<b \\
		\frac{1}{2}(b-a) & \quad \text{if } \, s \geq b \, .
	\end{cases}
	\]
	Then using $\varphi = \hat u\psi$ as a test function for the equation $\Delta_g u = 0$ we get
	\begin{align*}
		0 & = - \int_\Sigma \varphi \Delta_g u \, \di v_g \\
		& = \int_\Sigma \langle\nabla\varphi,\nabla u\rangle \, \di v_g \\
		& = \int_\Sigma \psi \|\nabla u\|^2 \mathbf{1}_{\{a<u<b\}} \, \di v_g + \int_\Sigma \hat u \langle\nabla\psi,\nabla u\rangle \, \di v_g \, .
	\end{align*}
	Using $\|\nabla u\|^2 = 1-W^{-2}$ and $\di v_g = W\di v$ we have
	\begin{align*}
		\int_\Sigma \psi\|\nabla u\|^2 \mathbf{1}_{\{a<u<b\}} \, \di v_g & = \int_\Sigma \psi \mathbf{1}_{\{a<u<b\}} \, \di v_g - \int_M \frac{\psi}{W} \mathbf{1}_{\{a<u<b\}} \, \di v \\
		& \geq \vol_g(\Sigma_{R,a,b}) - \vol(B_{R_1}) \, .
	\end{align*}
	On the other hand, since $\langle\nabla\psi,\nabla u\rangle = g^{ij}\psi_i u_j = W^{-2}\sigma^{ij}\psi_i u_j = W^{-2}(D\psi,D u)$ by \eqref{nabla_pu}, we also have
	\begin{align*}
		\int_\Sigma \hat u \langle\nabla\psi,\nabla u\rangle \, \di v_g & = \int_M \hat u \frac{(D\psi,Du)}{W} \, \di v \\
		& \geq - \int_M |\hat u||D\psi| \,\di v \\
		& \geq - \frac{b-a}{2} \frac{\vol(B_{R_1}\setminus B_R)}{R_1-R}
	\end{align*}
	and putting together all estimates we obtain
	\[
		\vol_g(\Sigma_{R,a,b}) \leq \vol(B_{R_1}) + \frac{b-a}{2} \frac{\vol(B_{R_1}\setminus B_R)}{R_1-R}
	\]
	that is, \eqref{Vest0}. Inequality \eqref{Vest1} can then be deduced from \eqref{Vest0} choosing $R_1 = R+\frac{1}{2}(b-a)$ and using that $\vol(B_{R_1}\setminus B_R) = \vol(B_{R_1})-\vol(B_R)$, while \eqref{Vest2} follows from \eqref{Vest0} since for a.e.~$R>0$ the limit
	\[
		\lim_{\eps\to0} \frac{\vol(B_{R+\eps})-\vol(B_R)}{\eps}
	\]
	exists and equals $\area(\partial B_R)$.
\end{proof}

The next Lemma elaborates on arguments in \cite[p. 264]{bdgm}.

\begin{lemma} \label{lem_eCu}
	Let $(M,\sigma)$ be a complete Riemannian manifold, $u\in C^\infty(M)$ a solution to \eqref{MSE} and $o\in M$ a fixed origin. For any $C>0$, $\lambda>0$ and $R>0$ we have
	\begin{equation} \label{Iest}
		\int_{\Sigma_R} e^{-2Cu} \, \di v_g \leq \exp\left(-2C\inf_{B_R} u\right) \frac{\vol(B_{(1+\lambda)R})}{1-e^{-4\lambda CR}}
	\end{equation}
	and for a.e.~$R>0$ we also have
	\begin{equation} \label{Iest'}
		\int_{\Sigma_R} e^{-2Cu} \, \di v_g \leq \exp\left(-2C\inf_{B_R} u\right) \frac{\vol(B_R)+\lambda R\,\area(\partial B_R)}{1-e^{-4\lambda CR}} \, .
	\end{equation}
\end{lemma}

\begin{proof}
	Noting that
	\begin{equation} \label{Iest0}
			\int_{\Sigma_R} e^{-2Cu} \, \di v_g = \exp\left(-2C\inf_{B_R}u\right) \int_{\Sigma_R} e^{-2C(u-\inf_{B_R} u)} \, \di v_g
	\end{equation}
	we can set $v = u-\inf_{B_R} u$ and then we are left with the problem of estimating $\int_{\Sigma_R} e^{-2Cv} \, \di v_g$. Since $v\geq 0$ on $B_R$, we have
	\[
		\Sigma_R = \bigcup_{n=0}^\infty \tilde\Sigma_{R,n}
	\]
	where, for each integer $n\geq 0$,
	\[
		\tilde\Sigma_{R,n} := \{ (x,u(x)) : x \in B_R , \, 2n\lambda R \leq v(x) < 2(n+1)\lambda R \} \, .
	\]
	Then
	\begin{equation} \label{intS}
		\int_{\Sigma_R} e^{-2Cv} \, \di v_g = \sum_{n=0}^\infty \int_{\tilde\Sigma_{R,n}} e^{-2Cv} \, \di v_g \leq \sum_{n=0}^\infty e^{-4n\lambda CR} \, \vol_g(\tilde\Sigma_{R,n}) \, .
	\end{equation}
	Note that for any $n\geq0$ we have
	\[
		\tilde{\Sigma}_{R,n} \subseteq \Sigma_{R,a,b}
	\]
	for all $a<2n\lambda R-\inf_{B_R}u$ and for $b=2(n+1)\lambda R-\inf_{B_R}u$. Using \eqref{Vest1} to bound $\vol_g(\Sigma_{R,a,b})$ from above and then letting $a\nearrow 2n\lambda R-\inf_{B_R}u$ we obtain
	\[
		\vol_g(\tilde\Sigma_{R,n}) \leq \vol(B_{(1+\lambda)R}) \qquad \forall \, n \geq 0 \, .
	\]
	Substituting this estimate into the RHS of \eqref{intS} we get
	\[
		\int_{\Sigma_R} e^{-2Cv} \, \di v_g \leq \frac{\vol(B_{(1+\lambda)R})}{1-e^{-4\lambda CR}}
	\]
	and then \eqref{Iest} follows from this together with \eqref{Iest0}. A similar reasoning using \eqref{Vest2} instead of \eqref{Vest1} yields \eqref{Iest'}.
\end{proof}

\begin{lemma} \label{lem_cac}
	Let $M$ be a complete Riemannian manifold with $\Ricc\geq 0$ and $u\in C^\infty(M)$ a solution to  \eqref{MSE}. For any $C>0$, $s_1>0$, $\alpha\geq 2$ and $\psi\in \lip_c(M)$, $\psi\geq 0$ we have
	\begin{equation} \label{int0}
		\frac{4C^2}{\alpha^2} \int_{\{z>s_1\}} \|\nabla u\|^2 \psi^\alpha e^{-2Cu} \, \di v_g \leq \int_{\{z>s_1\}} \psi^{\alpha-2}\|\nabla\psi\|^2 e^{-2Cu} \, \di v_g
	\end{equation}
	where $z = W e^{-Cu}$.
\end{lemma}

\begin{proof}
	Let $C>0$, $s_1>0$ and $\alpha\geq 2$ be given. By standard density arguments, it is enough to prove \eqref{int0} for smooth $\psi$, so let $0\leq\psi\in C^\infty_c(M)$ be given. We already observed in \eqref{divWz} that, since $M$ has non-negative Ricci curvature, the function $z$ satisfies
	\[
		\div_g(W^{-2}\nabla z) \geq C^2 \|\nabla u\|^2 W^{-2} z \, ,
	\]
	hence
	\begin{equation} \label{intb}
		C^2 \int_M \|\nabla u\|^2 W^{-2} z \varphi \, \di v_g \leq - \int_M W^{-2} \langle \nabla z,\nabla\varphi \rangle \, \di v_g
	\end{equation}
	for every $\varphi \in \lip_c(M)$, $\varphi\geq 0$. Let $\lambda\in C^\infty(\R)$ be such that
	\[
		\lambda(s) = 0 \quad \text{for } \, s \leq 1 \, , \qquad \lambda(s) = 1 \quad \text{for } \, s \geq 2 \, , \qquad \lambda'\geq 0 \qquad \text{on } \, \R
	\]
	and for each $\delta>0$ define $\lambda_\delta\in C^\infty(\R)$ by
	\[
		\lambda_\delta(s) := \lambda(s/\delta) \qquad \forall \, s \in \R \, .
	\]
	We have $\lambda_\delta\nearrow\mathbf{1}_{(0,\infty )}$ as $\delta\to0^+$. For each $\delta>0$ set
	\[
		\varphi_\delta := \psi^\alpha z \lambda_\delta(z-s_1) \, .
	\]
	We have $\varphi_\delta\in\lip_c(\Sigma)$, $\varphi_\delta\geq 0$, hence $\varphi_\delta$ is an admissible test function and by \eqref{intb}
	\begin{equation} \label{intb1}
		C^2 \int_{\{z>s_1\}} \|\nabla u\|^2 W^{-2} z^2 \psi^\alpha \lambda_\delta(z-s_1) \, \di v_g \leq - \int_{\{z>s_1\}} W^{-2} \langle \nabla z,\nabla\varphi_\delta \rangle \, \di v_g
	\end{equation}
	where we can restrict ourselves to integrating over $\{z>s_1\}$ since the support of $\varphi_\delta$ is contained in the closure of $\{\lambda_\delta(z-s_1) > 0\} \equiv \{z-s_1>\delta\}$, hence inside $\{z\geq s_1+\delta\}\subseteq\{z>s_1\}$. By direct computation and using $\lambda_\delta'\geq 0$ we get
	\begin{align*}
		\langle \nabla z,\nabla\varphi_\delta \rangle & = \alpha\psi^{\alpha-1}\lambda_\delta(z-s_1) z\langle\nabla z,\nabla\psi\rangle \\
		& \phantom{=\;} + \psi^\alpha \lambda_\delta(z-s_1)\|\nabla z\|^2 \\
		& \phantom{=\;} + \psi^\alpha \lambda_\delta'(z-s_1)z\|\nabla z\|^2 \\
		& \geq \lambda_\delta(z-s_1)\left(\alpha\psi^{\alpha-1} z\langle\nabla z,\nabla\psi\rangle + \psi^\alpha\|\nabla z\|^2\right)
	\end{align*}
	and by Young's inequality we have
	\[
		\alpha\psi^{\alpha-1} z\langle\nabla z,\nabla\psi\rangle + \psi^\alpha\|\nabla z\|^2 \geq -\frac{\alpha^2}{4} \psi^{\alpha-2} z^2 \|\nabla\psi\|^2
	\]
	hence, since $\lambda_\delta\geq 0$,
	\[
		- \langle \nabla z,\nabla\varphi_\delta \rangle \leq \frac{\alpha^2}{4} z^2 \psi^{\alpha-2} \|\nabla\psi\|^2 \lambda_\delta(z-s_1) \, .
	\]
	Substituting this into \eqref{intb1} we obtain
	\[
		C^2 \int_{\{z>s_1\}} \|\nabla u\|^2 W^{-2} z^2 \psi^\alpha \lambda_\delta(z-s_1) \, \di v_g \leq \frac{\alpha^2}{4} \int_{\{z>s_1\}} W^{-2} z^2 \psi^{\alpha-2} \|\nabla\psi\|^2 \lambda_\delta(z-s_1) \, \di v_g \, .
	\]
	By construction we have $W^{-2} z^2 = e^{-2Cu}$, hence this amounts to
	\[
		C^2 \int_{\{z>s_1\}} \|\nabla u\|^2 \psi^\alpha e^{-2Cu} \lambda_\delta(z-s_1) \, \di v_g \leq \frac{\alpha^2}{4} \int_{\{z>s_1\}} \psi^{\alpha-2} e^{-2Cu} \|\nabla\psi\|^2 \lambda_\delta(z-s_1) \, \di v_g
	\]
	and since $\lambda_\delta(z-s_1) \nearrow \mathbf{1}_{\{z>s_1\}}$, by monotone convergence when letting $\delta\searrow 0$ we obtain
	\[
		C^2 \int_{\{z>s_1\}} \|\nabla u\|^2 \psi^\alpha e^{-2Cu} \, \di v_g \leq \frac{\alpha^2}{4} \int_{\{z>s_1\}} \psi^{\alpha-2} e^{-2Cu} \|\nabla\psi\|^2 \, \di v_g
	\]
	that is, \eqref{int0}.
\end{proof}

\section{Proof of the main theorem}

Let $(M,\sigma)$ be a connected, complete Riemannian manifold with $\Ricc\geq 0$. As in the previous section, let $o\in M$ be a fixed origin and set $r(x)=\dist_\sigma(o,x)$ for each $x\in M$. Let $u\in C^\infty(M)$ be a solution of the minimal hypersurface equation \eqref{MSE} satisfying
\begin{equation} \label{u-bound}
	u \geq - f(r) \qquad \text{on } \, M
\end{equation}
for some continuous, positive and non-decreasing function $f : \R^+_0 \to \R^+$, where we adopt the notation
\[
	\R^+ = (0,\infty ) \, , \qquad \R^+_0 = [0,\infty ) \, .
\]
We assume that
\begin{equation} \label{f_cond}
	\lim_{t\to\infty } f(t) = \infty \, , \qquad \exists \, \lim_{t\to\infty } \frac{f(t)}{t} =: K \in [0,\infty ) \, .
\end{equation}
We claim that under these assumptions it is possible to choose $h : \R^+ \to \R^+$ continuous and positive in such a way that the following requirements are satisfied for some parameters $A\in\R^+_0$ and $t_0\in\R^+_0$:
\begin{equation} \label{hft0}
	\left\{
	\begin{array}{r@{\;}c@{\;}ll}
		th(t) & > & 2 & \qquad \text{for each } \, t > t_0 \\[0.2cm]
		h(t) & \to & 0 & \qquad \text{as } \, t \to \infty  \\[0.2cm]
		f(t)h(t) & \to & A & \qquad \text{as } \, t \to \infty  \, .
	\end{array}
	\right.
\end{equation}
Indeed, if $f(t) = o(t)$ (that is, $K=0$) then these conditions are satisfied with $A=0$ by
\[
	h(t) = \frac{1}{\sqrt{tf(t)}} + \frac{2}{t}
\]
(in particular, in this case $th(t) = 2+ \sqrt{t/f(t)} \to \infty $ as $t\to\infty $); otherwise, that is $K>0$, the conditions in \eqref{hft0} are satisfied for instance by
\[
	h(t) = \frac{2}{t} + \frac{1}{t^2}
\]
with $A=2K$. For any $h$ satisfying the requirements in \eqref{hft0}, we set
\begin{equation} \label{ell_def}
	\ell = \liminf_{t\to\infty } th(t) \in [2,\infty ] \, .
\end{equation}

Let $M$, $u$ and $f$ be as above and let $h$ satisfy \eqref{hft0} for some $t_0,A\in\R^+_0$. Fix two parameters
\[
	s_0 > 1 \, , \qquad \mu > 0 \, ,
\]
set
\begin{equation} \label{Lambda}
	\Lambda = \frac{e^\mu}{\sqrt{1-s_0^{-2}}}
\end{equation}
and for each $R>0$ define
\begin{equation} \label{CzR}
	C_R = \Lambda h(R) \, , \qquad z_R = W e^{-C_R u} \, , \qquad s_1(R) = s_0 \exp(\Lambda f(R)h(R))
\end{equation}
and
\begin{equation} \label{ER_def}
	E_R = \{ x \in B_R : z_R(x) > s_1(R) \} \, .
\end{equation}
where $B_R = B_R^\sigma(o)$ as defined in the previous section. We claim that the following inclusion holds:
\begin{equation} \label{inc1g}
	E_R \subseteq \{ x \in B_R : W(x) > s_0\} \qquad \forall \, R>0 \, .
\end{equation}
Indeed, for any $R>0$ and for any $x\in B_R$ such that $W(x)\leq s_0$ we have
\begin{align*}
	z_R(x) & = W(x)\exp(-\Lambda u(x)h(R)) \\
	& \leq s_0\exp(-\Lambda u(x)h(R)) \\
	& \leq s_0\exp(\Lambda f(r(x))h(R)) \\
	& \leq s_0\exp(\Lambda f(R)h(R)) \\
	& = s_1(R) \, ,
\end{align*}
where we used positivity of $s_0$, $\Lambda$, $h$, $f$, monotonicity of $f$ and assumption \eqref{u-bound}. This shows that
\[
\{x\in B_R : W(x) \leq s_0\} \subseteq \{ x \in B_R : z_R(x) \leq s_1(R) \} \qquad \forall R > 0
\]
and passing to the complements in $B_R$ we get \eqref{inc1g}.

\begin{lemma} \label{lem_base}
	Let $M$ be a connected, complete Riemannian manifold with $\Ricc\geq 0$ and $u\in C^\infty(M)$ a solution to \eqref{MSE} satisfying \eqref{u-bound} for some non-decreasing $f:\R_0^+ \to \R^+$. Let $h$, $s_0$, $\mu$, $\Lambda$, $C_R$, $z_R$ and $s_1(R)$ be as above. Then
	\begin{equation} \label{g1}
		\int_{E_R\cap B_{R/2}} e^{-2C_R u} \, \di v_g \leq e^{-\mu R h(R)} \vol(B_{2R}) \left(\frac{e^{2\Lambda A}}{1-e^{-4\Lambda\ell}} + o(1)\right) \qquad \text{as } \, R \to \infty 
	\end{equation}
	(with the agreement that $e^{-\infty}=0$ in case $\ell=\infty $). Moreover, if
	\begin{equation} \label{g1b}
		\vol(B_{2R}) = o(e^{\mu Rh(R)}) \qquad \text{as } \, R \to \infty 
	\end{equation}
	then
	\begin{equation} \label{g1c}
		\sup_M W \leq s_0 e^{A\Lambda} < \infty  \, .
	\end{equation}
\end{lemma}

\begin{proof}[Proof of Lemma \ref{lem_base}]
	We start by proving \eqref{g1}. For each $R>0$ we apply Lemma \eqref{lem_cac} with parameters $C=C_R$, $s_1=s_1(R)$ and with $\psi=\psi_R$ the cut-off function given by
	\[
	\psi_R = \begin{cases}
		1 & \quad \text{on } \, B_{R/2} \\
		\dfrac{2R-2r}{R} & \quad \text{on } \, B_R\setminus B_{R/2} \\
		0 & \quad \text{on } \, M \setminus B_R
	\end{cases}
	\]
	to obtain
	\begin{equation} \label{in0}
		\frac{4C_R^2}{\alpha^2} \int_{E_R} \|\nabla u\|^2 \psi_R^\alpha e^{-2C_R u} \, \di v_g \leq \int_{E_R} \psi_R^{\alpha-2} \|\nabla\psi_R\|^2 e^{-2C_R u} \, \di v_g \qquad \forall \, \alpha \geq 2 \, ,
	\end{equation}
	where we restricted the integration domain to $E_R$ since $\psi_R$ and $\|\nabla\psi_R\|$ vanish a.e.~outside $B_R$. By \eqref{inc1g} we have
	\[
		W > s_0 \qquad \text{on } \, E_R \, .
	\]
	Since $\|\nabla u\|^2 = 1 - W^{-2}$, this yields
	\[
		\|\nabla u\|^2 > 1 - s_0^{-2} \qquad \text{on } \, E_R
	\]
	and then by \eqref{CzR} and \eqref{Lambda}
	\begin{equation} \label{CRDu}
		C_R^2 \|\nabla u\|^2 = \Lambda^2\|\nabla u\|^2 h(R)^2 > \Lambda^2(1-s_0^{-2}) h(R)^2 = e^{2\mu} h(R)^2 \qquad \text{on } \, E_R \, .
	\end{equation}
	On the other hand, we have
	\begin{equation} \label{DpsiR}
		\|\nabla\psi_R\|^2 \leq |D\psi_R|^2 \leq \frac{4}{R^2} \qquad \text{on } \, M
	\end{equation}
	Substituting \eqref{CRDu} and \eqref{DpsiR} into \eqref{in0} we obtain
	\begin{equation} \label{in1}
		\frac{e^{2\mu} R^2 h(R)^2}{\alpha^2} \int_{E_R} \psi_R^\alpha e^{-2C_R u} \, \di v_g \leq \int_{E_R} \psi_R^{\alpha-2} e^{-2C_R u} \, \di v_g \qquad \forall \, \alpha \geq 2 \, .
	\end{equation}
	For any $\alpha>2$ an application of H\"older's inequality with conjugate exponents $\frac{\alpha}{\alpha-2}$ and $\frac{\alpha}{2}$ yields
	\[
		\int_{E_R} \psi_R^{\alpha-2} e^{-2C_R u} \, \di v_g \leq \left( \int_{E_R} \psi_R^\alpha e^{-2C_R u} \, \di v_g \right)^{1-\frac{2}{\alpha}} \left( \int_{E_R} e^{-2C_R u} \, \di v_g \right)^{\frac{2}{\alpha}} \, .
	\]
	Substituting this into \eqref{in1} and then raising everything to the exponent $\frac{\alpha}{2}$ we get
	\[
		\left( \frac{e^\mu R h(R)}{\alpha} \right)^{\alpha} \int_{E_R} \psi_R^\alpha e^{-2C_R u} \, \di v_g \leq \int_{E_R} e^{-2C_R u} \, \di v_g \qquad \forall \, \alpha > 2 \, .
	\]
	Recalling that $\psi_R\equiv 1$ on $B_{R/2}$, we further obtain
	\begin{equation} \label{in2}
		\left( \frac{e^\mu R h(R)}{\alpha} \right)^{\alpha} \int_{E_R\cap B_{R/2}} e^{-2C_R u} \, \di v_g \leq \int_{E_R} e^{-2C_R u} \, \di v_g \qquad \forall \, \alpha > 2 \, .
	\end{equation}
	By the first assumption in \eqref{hft0}, for each $R>t_0$ we apply \eqref{in2} with the choice
	\[
		\alpha = \alpha(R) := R h(R) > 2
	\]
	to obtain
	\begin{equation} \label{in3}
		e^{\mu Rh(R)} \int_{E_R\cap B_{R/2}} e^{-2C_R u} \, \di v_g \leq \int_{E_R} e^{-2C_R u} \, \di v_g \qquad \forall \, R > t_0 \, .
	\end{equation}
	By Lemma \ref{lem_eCu} applied with the choice $C=C_R\equiv \Lambda h(R)$ and $\lambda=1$, and by also recalling \eqref{u-bound}, we estimate
	\begin{align*}
		\int_{E_R} e^{-2C_R u} \, \di v_g \leq \int_{\Sigma_R} e^{-2C_R u} \, \di v_g & \leq \frac{\vol(B_{2R})}{1-e^{-4C_R R}} \exp\left(-2C_R \inf_{B_R} u\right) \\
		& \leq \frac{\vol(B_{2R})}{1-e^{-4\Lambda Rh(R)}} \exp\left(2\Lambda f(R)h(R)\right)
	\end{align*}
	and thus we finally obtain
	\begin{equation} \label{in4}
		\int_{E_R\cap B_{R/2}} e^{-2C_R u} \, \di v_g \leq \frac{e^{-\mu Rh(R)}}{1-e^{-4\Lambda Rh(R)}} \vol(B_{2R}) \exp\left(2\Lambda f(R)h(R)\right) \qquad \forall \, R > t_0 \, .
	\end{equation}
	From the third condition in \eqref{hft0} and the definition of $\ell := \liminf_{t\to\infty } th(t)$ we have
	\[
		\limsup_{R\to\infty } \frac{\exp\left(2\Lambda f(R)h(R)\right)}{1-e^{-4\Lambda Rh(R)}} = \frac{e^{2\Lambda A}}{1-e^{-4\Lambda\ell}} \, ,
	\]
	that is,
	\[
		\frac{\exp\left(2\Lambda f(R)h(R)\right)}{1-e^{-4\Lambda Rh(R)}} \leq \frac{e^{2\Lambda A}}{1-e^{-4\Lambda\ell}} + o(1) \qquad \text{as } \, R \to \infty 
	\]
	and substituting this into \eqref{in4} we obtain \eqref{g1}.
	
	We now prove the last part of the statement. Suppose by contradiction that \eqref{g1b} holds but \eqref{g1c} is not satisfied. Then there exists $x_0\in M$ such that $W(x_0) > s_0 e^{A\Lambda}$ and by continuity we can find $\gamma>s_0 e^{A\Lambda}$ and a relatively compact open neighbourhood $\Omega$ of $x_0$ such that
	\[
		W > \gamma \qquad \text{on } \, \Omega \, .
	\]
	Since $u$ is bounded on $\Omega$ and $C_R = \Lambda h(R)\to 0$ as $R\to\infty $ by \eqref{hft0}, we have
	\[
		e^{-C_R u} \to 1 \qquad \text{and therefore} \qquad z_R \to W \qquad \text{uniformly on $\Omega$ as } \, R \to \infty  \, .
	\]
	On the other hand, we also have
	\[
		s_1(R) \to s_0 e^{A\Lambda} < \gamma \qquad \text{as } \, R \to \infty  \, .
	\]
	Hence, setting $\eps = \frac{1}{2}(\gamma-s_0 e^{A\Lambda})$, there exists $R_1>0$ such that
	\[
		z_R > W - \eps > \gamma - \eps = s_0 e^{A\Lambda} + \eps > s_1(R) \qquad \text{on } \, \Omega \qquad \forall \, R > R_1 \, ,
	\]
	that is, $\Omega \subseteq E_R$ for all $R>R_1$. Since $\Omega$ is bounded, up to choosing a larger $R_1$ we can also assume that $\Omega \subseteq E_R \cap B_{R/2}$ for all $R>R_1$. But then from \eqref{g1} and \eqref{g1b} we have
	\[
		0 = \lim_{R\to\infty } \int_{E_R\cap B_{R/2}} e^{-2C_R u} \, \di v_g \geq \lim_{R\to\infty } \int_\Omega e^{-2C_R u} \, \di v_g = \vol_g(\Omega) > 0 \, ,
	\]
	contradiction.
\end{proof}

We next need the following well-known Lemma. We provide a detailed proof for the sake of completeness.
 
\begin{lemma} \label{lem_liou}
	Let $M$ be a complete Riemannian manifold with $\Ricc\geq 0$, and $u\in C^\infty(M)$ a solution to \eqref{MSE} such that
	\[
		(i) \; \; \sup_M |Du| < \infty  \qquad \text{and} \qquad (ii) \; \; u_- (x) = o(r(x)) \quad \text{as } \, r(x) \to \infty \, .
	\]
	Then $u$ is constant.
\end{lemma}

\begin{proof}
	A way to see this is to use \cite[Theorem 8]{cgmr} and its proof, according to which, if $u$ is non-constant, any blowdown $u_\infty : M_\infty \to \R$ of $u$ corresponding to the tangent cone $M_\infty = N_\infty \times \R$ satisfies $u_\infty(y,t) = \|Du\|_\infty t$. In particular, $u$ has exactly linear growth both from above and from below. We provide a more direct argument. Without loss of generality we can assume that $u(o)=0$, so that
	\[
		\inf_{B_R} u \leq 0 \qquad \forall \, R > 0 \, .
	\]
	From assumption (i), the operator $L = W\Delta_g$ is a uniformly elliptic operator in divergence form on the complete manifold $(M,\sigma)$ of non-negative Ricci curvature. Since $L u = 0$, for each $R>0$ the function
	\[
		u_R = u - \inf_{B_{2R}} u
	\]
	satisfies $L u_R = 0$ on $M$, $u_R \geq 0$ on $B_{2R}$ and
	\[
		\inf_{B_R} u_R = \inf_{B_R} u - \inf_{B_{2R}} u \leq -\inf_{B_{2R}} u \, .
	\]
	Since $u_R$ is non-negative on $B_{2R}$, from the Harnack inequality of \cite[Theorem 5.3]{sc92} we have
	\[
		\sup_{B_R} u_R \leq C \inf_{B_R} u_R \leq - C \inf_{B_{2R}} u
	\]
	with $C>0$ constant independent of $R$, and thus
	\[
		\sup_{B_R} u = \sup_{B_R} u - \inf_{B_{2R}} u \leq - (1+C) \inf_{B_{2R}} u = o(R) \qquad \text{as } \, R \to \infty  \, .
	\]
	From this and (ii) we obtain that $|u(x)| = o(r(x))$ as $x\to\infty$, and since $|Du|\in L^\infty(M)$ we can apply the argument from the proof of \cite[Theorem 3.6]{djx16} and \cite[Theorem 11]{cgmr} to conclude that $u$ is constant. We briefly sketch the argument for reader's convenience: since $Lu=0$ and $L$ is uniformly elliptic, we have a Caccioppoli inequality
	\[
		\int_M \psi^2 |Du|^2 \, \di v \leq C \int_M u^2 |D\psi|^2 \, \di v \qquad \forall \, \psi \in \lip_c(M)
	\]
	with $C>0$ a fixed constant. Fix $\eps>0$. Since $|u| = o(r)$, there exists $R_0>0$ large enough so that $u^2 \leq \eps R^2$ on $B_{2R}$ for all $R>R_0$. Applying the Caccioppoli inequality with the cut-off function
	\[
		\psi_R = \left\{
			\begin{array}{ll}
				1 & \text{on } \, B_R \\
				\dfrac{2R-r}{R} & \text{on } \, B_{2R}\setminus B_R \\
				0 & \text{on } \, M \setminus B_{2R}
			\end{array}
		\right.
	\]
	we get
	\[
		\int_{B_R} |Du|^2 \, \di v \leq C\eps\,\vol(B_{2R}) \leq 2^m C \eps\,\vol(B_R) \qquad \forall \, R>R_0
	\]
	where we used Bishop-Gromov's inequality with $m=\dim M$. Letting $R\to\infty $ (with $\eps>0$ fixed) we get
	\begin{equation} \label{li1}
		\lim_{R\to\infty } \frac{1}{\vol(B_R)} \int_{B_R} |Du|^2 \, \di v \leq 2^m C \eps \, .
	\end{equation}
	The function $|Du|^2$ is bounded and satisfies
	\[
		L |Du|^2 = W\Delta_g W^2 \geq W^2 \Delta_g W \geq (\|\II\|^2 + \Ricc_{\bar\sigma}(\mathbf{n},\mathbf{n})) W^3 \geq 0
	\]
	due to Jacobi equation \eqref{Jac} and condition $\Ricc\geq 0$. Again since $L$ is a uniformly elliptic operator in divergence form on a manifold with $\Ricc\geq 0$, we apply Proposition 22 in \cite{cgmr} to the non-negative, bounded, $L$-superharmonic function $f=\sup_M|Du|^2 - |Du|^2$ to obtain the Li-type mean value formula
	\begin{equation} \label{li2}
		\sup_M |Du|^2 = \lim_{R\to\infty } \frac{1}{\vol(B_R)} \int_{B_R} |Du|^2 \, \di v \, .
	\end{equation}
	Combining \eqref{li1} and \eqref{li2} we get
	\[
		\sup_M |Du|^2 \leq 2^m C \eps \qquad \forall \, \eps > 0
	\]
	that is, $Du\equiv 0$ on $M$. This concludes the proof.
\end{proof}

We can now prove Theorem \ref{thm_main2} from the Introduction. 

\begin{proof}[Proof of Theorem \ref{thm_main2}] 
Fix $A>0$ such that 
	\[
		u \geq - \frac{Ar}{\log(1+r)} \qquad \text{on } \, M,
	\]
and consider the functions
	\[
		f(t) = \frac{At}{\log(1+t)} \, , \qquad h(t) = \frac{3+\log(1+t)}{t},
	\]
	which satisfy all assumptions in \eqref{hft0}. By the Bishop-Gromov comparison theorem, 
	\[
		\vol(B_{2R}) \leq C R^m \qquad \forall \, R > 0
	\]
	with $m=\dim M$ and $C=C(m)>0$ a constant independent of $R$. Therefore, for any $\mu>0$ we have
	\[
		e^{-\mu Rh(R)}\vol(B_{2R}) \leq (e^{-\mu})^{3+\log(1+R)} C R^m = \frac{C}{e^{3\mu}} \frac{R^m}{(R+1)^\mu} \qquad \forall \, R > 0 \, .
	\]
	Thus, for any choice of parameters $\mu>m$ and $s_0>1$ we deduce from Lemma \ref{lem_base} that
	\[
		\sup_M W \leq s_0 e^{A\Lambda} < \infty 
	\]
	with $\Lambda = \Lambda(\mu,s_0)$ defined as in \eqref{Lambda}. In particular, $|Du| = \sqrt{W^2-1}$ is bounded on $M$. By the previous Lemma \ref{lem_liou}, we conclude that $u$ is constant.
\end{proof}

\section{Further results}

Let $(M,\sigma)$ be a connected, complete Riemannian manifold with $\Ricc\geq 0$ and let $r$, $B_R$ denote, respectively, the Riemannian distance from a fixed origin $o\in M$ and the geodesic ball of radius $R>0$ centered at $o$. In this section, we derive asymptotic (as $R\to\infty $) upper estimates on the measure of the sets $\{|Du|>\eps\}\cap B_R$, $\eps>0$, in case $u$ is a solution to \eqref{MSE} on $M$ satisfying $u=o(r)$. In the last part of the section we show that condition $u=o(r)$ can in fact be relaxed to $u_- = o(r)$, Therefore, the estimates proved in this section can be seen as complementary, weaker results to those proved in the previous section, where the stronger assumption $u_- = O(r/\log r)$ was required.

Let $(M,\sigma)$ and $r$ be as above and let $u\in C^\infty(M)$ be a solution to \eqref{MSE} satisfying
\begin{equation} \label{|u|bound}
	|u| \leq f(r) \qquad \text{on } \, M
\end{equation}
for a continuous, positive and non-decreasing $f : \R^+_0 \to \R^+$. Let $h,\ell,s_0,\mu,\Lambda,C_R,z_R,s_1(R),E_R$ be as at the beginning of the previous section. For each $R>0$ define
\[
	s_2(R) = s_0\exp(2\Lambda f(R)h(R))
\]
and
\[
	\Omega_R = \{ x \in B_R : W(x) > s_2(R) \} \, .
\]
We observe that, due to the upper bound $u\leq f(r)$ ensured by \eqref{|u|bound}, the following inclusion holds:
\begin{equation} \label{OER}
	\Omega_R \subseteq E_R \qquad \forall \, R > 0 \, .
\end{equation}
Indeed, for any $R>0$ and $x\in B_R$ such that $W(x) > s_2(R)$ we have
\begin{align*}
	z_R(x) & = W(x) \exp(-\Lambda u(x)h(R)) \\
	& > s_2(R) \exp(-\Lambda u(x)h(R)) \\
	& \geq s_2(R) \exp(-\Lambda f(r(x))h(R)) \\
	& \geq s_2(R) \exp(-\Lambda f(R)h(R)) \\
	& = s_0 \exp(\Lambda f(R)h(R)) \\
	& = s_1(R)
\end{align*}
where in the second inequality we used that $u(x)\leq f(r(x))$.

\begin{lemma}
	Let $M$ be a connected, complete Riemannian manifold with $\Ricc\geq 0$ and $u\in C^\infty(M)$ a solution to \eqref{MSE} satisfying \eqref{|u|bound}. In the above setting,
	\begin{equation} \label{g2}
		\int_{\Omega_R\cap B_{R/2}} e^{-2C_R u} \leq e^{-\mu R h(R)} \vol(B_{2R}) \left(\frac{e^{2\Lambda A}}{1-e^{-4\Lambda\ell}} + o(1)\right) \qquad \text{as } \, R \to \infty  \, .
	\end{equation}
	In particular, for each
	\begin{equation} \label{g2b}
		\gamma > s_0 e^{2A\Lambda}
	\end{equation}
	we have
	\begin{equation} \label{g2c}
		\vol_g(\{W>\gamma\}\cap B_{R/2}) \leq e^{-\mu R h(R)} \vol(B_{2R}) \left(\frac{e^{2\Lambda A}}{1-e^{-4\Lambda\ell}} + o(1)\right) \qquad \text{as } \, R \to \infty  \, .
	\end{equation}
\end{lemma}

\begin{proof}
	Inequality \eqref{g2} follows directly from \eqref{g1} coupled with inclusion \eqref{OER}. The second part of the statement then follows by noting that for any $\gamma>s_0 e^{2A\Lambda}$ one has
	\[
		\{W>\gamma\} \subseteq \{W>s_2(R)\} \qquad \text{for all sufficiently large } \, R > 0
	\]
	since $s_2(R) \to s_0 e^{2A\Lambda} < \gamma$ as $R\to\infty $.
\end{proof}

A consequence of the previous Lemma is the following.

\begin{corollary} \label{cor_q1}
	Let $(M,\sigma)$ be a complete Riemannian manifold with $\Ricc\geq 0$ and $u\in C^\infty(M)$ a solution to \eqref{MSE} satisfying
	\[
		u(x) = o(r(x)) \qquad \text{as } \, x \to \infty \, .
	\]
	Then for each $\eps>0$ we have
	\[
		\vol_g(\{|Du|>\eps\}\cap B_R) = o(B_R) \qquad \text{as } \, R \to \infty  \, .
	\]
	In particular,
	\[
		\vol(\{|Du|>\eps\}\cap B_R) = o(B_R) \, .
	\]
\end{corollary}

\begin{proof}
	If $u(x) = o(r(x))$ as $x\to\infty$ then condition $|u|\leq f(r)$ is satisfied for a suitable continuous, non-decreasing function $f : \R^+_0 \to \R^+$ such that
	\[
		\left\{
		\begin{array}{r@{\;}c@{\;}ll}
			f(t) & = & o(t) & \qquad \text{as } \, t \to \infty  \\[0.2cm]
			f(t) & \to & \infty  & \qquad \text{as } \, t \to \infty 
		\end{array}
		\right.
	\]
	Consequently, as already observed at the beginning of the previous section, one can find $h : \R^+ \to \R^+$ satisfying the set of conditions \eqref{hft0} with $A=0$ (and $\ell=\infty $). Thus, for any choice of parameters $\gamma>s_0>1$ and $\mu>0$ estimate \eqref{g2c} yields
	\begin{align*}
		\vol_g(\{W>\gamma\}\cap B_{R/2}) & \leq e^{-\mu Rh(R)} \vol(B_{2R}) (1+o(1)) \\
		& \leq e^{-\mu Rh(R)} 4^m \vol(B_{R/2}) (1+o(1)) \\
		& = o(\vol(B_{R/2})) \qquad \qquad \text{as } \, R \to \infty 
	\end{align*}
	where we used Bishop-Gromov's theorem to estimate $\vol(B_{2R}) \leq 4^m \vol(B_{R/2})$ and $e^{-\mu Rh(R)} = o(1)$ as $R\to\infty $. Relabeling $R/2 \mapsto R$, we get
	\[
		\vol_g(\{W>\gamma\} \cap B_R) = o(\vol(B_R)) \qquad \text{as } \, R \to \infty 
	\]
	for each $\gamma>1$, which is equivalent to
	\[
		\vol_g(\{|Du|>\eps\} \cap B_R) = o(\vol(B_R)) \qquad \text{as } \, R \to \infty 
	\]
	for each $\eps>0$.
\end{proof}

In \cite{ding21}, Q. Ding proved a Harnack inequality for entire solutions of \eqref{MSE} on complete Riemannian manifolds satisfying the volume-doubling property and a uniform Neumann-Poincar\'e inequality. This class includes complete manifolds with non-negative Ricci curvature. By means of this Harnack inequality, that we restate in Lemma \ref{lemma_ding} below in the particular setting of manifolds with $\Ricc\geq 0$, we are able to prove (Corollary \ref{cor2h}) that for an entire solution $u$ of \eqref{MSE} the two-sided bound $u = o(r)$ appearing in the assumptions of Corollary \ref{cor_q1} is implied, in fact, by the one-sided bound $u_- = o(r)$.

\begin{lemma}[\cite{ding21}, Theorem 4.3] \label{lemma_ding}
	Let $M$ be a complete manifold with $\Ricc\geq 0$ and let $u\in C^\infty(M)$ be a solution to  \eqref{MSE}. Suppose that $u>0$ on $B_{4R}(p)$ for some $p\in M$ and $R>0$, and let
	\begin{equation} \label{DB2R}
		\mathfrak D_{2R} = \{(x,t)\in M\times\R : \dist_\sigma(x,p)+|t-u(p)| < 2R \} \, , \qquad \mathcal B_{2R} = \Sigma \cap \mathfrak D_{2R}(\bar x) \, .
	\end{equation}
	Then
	\begin{equation} \label{harnack_1}
		\sup_{\mathcal B_{2R}} u \leq \theta \inf_{\mathcal B_{2R}} u
	\end{equation}
	where $\theta>1$ is a constant depending only on $m=\dim M$.
\end{lemma}

\begin{corollary} \label{cor_harn}
	Let $M$ be a complete Riemannian manifold with $\Ricc\geq 0$ and let $u\in C^\infty(M)$ be a solution to  \eqref{MSE}. Suppose that $u>0$ on $B_{4R}(p)$ for some $p\in M$ and $R>0$. Then
	\begin{equation} \label{harnack}
		\min\left\{\sup_{B_R(p)} u, u(p)+R\right\} \leq \theta \max\left\{\inf_{B_R(p)} u, u(p)-R\right\}
	\end{equation}
	where $\theta>1$ is the constant appearing in Lemma \ref{lemma_ding}, depending only on $m=\dim M$.
\end{corollary}

\begin{proof}
	Let us consider
	\[
		\tilde\Gamma_R = \{ (x,t) \in M\times\R : \dist_\sigma(x,p) < R , \, |t-u(p)| < R \} \, , \qquad \tilde\Sigma_R = \Sigma \cap \tilde\Gamma_R \, .
	\]
	We have $\tilde\Gamma_R \subseteq \mathfrak D_{2R}$ and therefore $\tilde\Sigma_R\subseteq\mathcal B_{2R}$, where $\mathfrak D_{2R}$ and $\mathcal B_{2R}$ are as in \eqref{DB2R}. Hence,
	\begin{equation} \label{har1}
		\sup_{\tilde\Sigma_R} u \leq \sup_{\mathcal B_{2R}} u \leq \theta \inf_{\mathcal B_{2R}} u \leq \theta \inf_{\tilde\Sigma_R} u \, .
	\end{equation}
	When regarding $u$ as a function on the graph $\Sigma$, we have
	\[
		u(p) - R < u < u(p) + R \qquad \text{on } \, \tilde\Sigma_R
	\]
	by the very construction of $\tilde\Gamma_R$. On the other hand, again by the construction of $\tilde\Gamma_R$ we have $x\in B_R(p)$ for each $(x,t)\in\tilde\Gamma_R$. Therefore,
	\begin{align*}
		\inf_{\tilde\Sigma_R} u = \max\left\{ \inf_{B_R(p)} u, u(p)-R \right\} , \qquad \sup_{\tilde\Sigma_R} u = \min\left\{ \sup_{B_R(p)} u, u(p)+R \right\}
	\end{align*}
	and substituting this into \eqref{har1} we get \eqref{harnack}.
\end{proof}

\begin{corollary} \label{cor2h}
	Let $(M,\sigma)$ be a complete Riemannian manifold with $\Ricc\geq 0$ and let $u\in C^\infty(M)$ be a solution to \eqref{MSE}. If $u_-(x) = o(r(x))$ as $r(x)\to\infty$, then also $|u(x)| = o(r(x))$ as $r(x)\to\infty$. 
\end{corollary}

\begin{proof}
	For each $R>0$ let
	\[
		u_R = u - \inf_{B_{4R}} u + 1 \, .
	\]
	Then $u_R$ is a solution to \eqref{MSE} on $M$ that satisfies $u_R > 0$ on $B_{4R}$. By Corollary \ref{cor_harn}, for each $R>0$ we have
	\[
		\min\left\{ \sup_{B_R} u_R, u_R(o) + R \right\} \leq \theta \max\left\{ \inf_{B_R} u_R, u_R(o) - R \right\}
	\]
	Equivalently,
	\begin{equation} \label{harn2}
		\min\left\{ \sup_{B_R} u, u(o) + R \right\} \leq \theta \max\left\{ \inf_{B_R} u, u(o) - R \right\} + (\theta-1)\left(1 - \inf_{B_{4R}} u \right) \, .
	\end{equation}
	Since $\inf_{B_R} u = o(R)$ as $R\to\infty $, there exists $R_0>0$ such that
	\[
		\inf_{B_R} u > u(o) - R \qquad \forall \, R > R_0
	\]
	and therefore we have
	\begin{equation} \label{harn_or}
		\min\left\{ \sup_{B_R} u, u(o) + R \right\} \leq \theta \inf_{B_R} u + (\theta-1)\left(1 - \inf_{B_{4R}} u \right) \qquad \forall \, R > R_0 \, .
	\end{equation}
	Since the RHS of \eqref{harn_or} is $o(R)$ as $R\to\infty $, and therefore it is eventually smaller than $u(o)+R$, we infer that there exists $R_1>R_0$ such that
	\[
		\min\left\{ \sup_{B_R} u, u(o) + R \right\} = \sup_{B_R} u \qquad \forall \, R > R_1
	\]
	so we get
	\[
		\sup_{B_R} u \leq \theta \inf_{B_R} u + (\theta-1)\left(1 - \inf_{B_{4R}} u \right) \qquad \forall \, R > R_1
	\]
	and in particular $\sup_{B_R} u = o(R)$ as $R\to\infty$. This concludes the proof.
\end{proof}

\noindent \textbf{Acknowledgements.} L.M. and M.R. are supported by the PRIN project 20225J97H5 ``Differential-geometric aspects of manifolds via Global Analysis''.

\bibliographystyle{plain}

\begin{thebibliography}{99}
	
	\bibitem{bcmmpr} B. Bianchini, G. Colombo, M. Magliaro, L. Mari, P. Pucci and M. Rigoli, \emph{Recent rigidity results for graphs with prescribed mean curvature}, Math. Eng. \textbf{3} (2021), no. 5, Paper no. 039, 48 pp. MR4181196

	\bibitem{bmpr} B. Bianchini, L. Mari, P. Pucci and M. Rigoli, \emph{Geometric Analysis of Quasilinear Inequalities on complete manifolds. Maximum and compact support principles and detours on manifolds.} Frontiers in Mathematics, Birkh\"auser/Springer, Cham, (2021), 286 pp. MR4241012
	
	\bibitem{bdgm} E. Bombieri, E. De Giorgi and M. Miranda, \emph{Una maggiorazione a priori relativa alle ipersuperfici minimali non parametriche}, Arch. Rational Mech. Anal. \textbf{32} (1969), 255--267. MR0248647
	
	\bibitem{bg} E. Bombieri and E. Giusti, \emph{Harnack's inequality for elliptic differential equations on minimal surfaces}, Invent. Math. \textbf{15} (1972), 24--46. MR0308945

	\bibitem{ccm95} J. Cheeger, T.H. Colding and W.P. Minicozzi II, \emph{Linear growth harmonic functions on complete manifolds with nonnegative Ricci curvature}, Geom. Funct. Anal. \textbf{5} (1995), 948--954. MR1361516
	
	\bibitem{chengyau} S.Y. Cheng and S.T. Yau, \emph{Differential equations on Riemannian manifolds and their geometric applications}, Comm. Pure Appl. Math. \textbf{28} (1975), no. 3, 333--354. MR0385749

	\bibitem{cgmr} G. Colombo, E. S. Gama, L. Mari and M. Rigoli, \emph{Non-negative Ricci curvature and minimal graphs of linear growth}, to appear in Analysis \& PDE, preprint available at \href{https://arxiv.org/abs/2112.09886v1}{\texttt{arXiv:2112.09886}}
	
	\bibitem{cmmr} G. Colombo, M. Magliaro, L. Mari and M. Rigoli, \emph{Bernstein and half-space properties for minimal graphs under Ricci lower bounds}, Int. Math. Res. Not. IMRN \textbf{2022} (2022), no. 23, 18256--18290. MR4519145
	
	\bibitem{cmr21} G. Colombo, L. Mari and M. Rigoli, \emph{A splitting theorem for capillary graphs under Ricci lower bounds}, J. Funct. Anal. \textbf{281} (2021), article no. 109136, 50 pp. MR4271788
	
	\bibitem{ding21} Q. Ding, \emph{Liouville-type theorems for minimal graphs over manifolds}, Analysis \& PDE \textbf{14} (2021), no. 6, 1925--1949. MR4308670
	
	\bibitem{ding_new} Q. Ding, \emph{Poincar\'e inequality on minimal graphs over manifolds and applications}, preprint available at \href{https://arxiv.org/abs/2111.04458}{\texttt{arXiv:2111.04458}}
 
	\bibitem{djx16} Q. Ding, J. Jost and Y. Xin, \emph{Minimal graphic functions on manifolds of nonnegative Ricci curvature}, Comm. Pure Appl. Math. \textbf{69} (2016), no. 2, 323--371. MR3434614
	
	\bibitem{dfr10} N. do Esp\'irito-Santo, S. Fornari, J. B. Ripoll, \emph{The Dirichlet problem for the minimal hypersurface equation in $M\times\R$ with prescribed asymptotic boundary}, J. Math. Pures Appl. (9) \textbf{93} (2010), no. 2, 204--221. MR2584742
	
	\bibitem{dvh94} D. M. Duc, N. Van Hieu, \emph{Graphs with prescribed mean curvature on hyperbolic spaces}, Manuscripta Math. \textbf{83} (1994), no. 2, 111--121. MR1272177
	
	\bibitem{dvh95} D. M. Duc, N. Van Hieu, \emph{Graphs with prescribed mean curvature on Poincar\'e disk}, Bull. London Math. Soc. \textbf{27} (1995), no. 4, 353--358. MR1335286
	
	\bibitem{giusti} E. Giusti, \emph{Minimal surfaces and functions of bounded variation.} Monographs in Mathematics, 80, xii+240 pp. Birkh\"auser Verlag, Basel (1984). MR0775682

	\bibitem{kasue} A. Kasue, \emph{Harmonic functions with growth conditions on a manifold of asymptotically nonnegative curvature. II}, Recent topics in differential and analytic geometry, 283--301, Adv. Stud. Pure Math., 18-I, Academic Press, Boston, MA, 1990. MR1145260

	\bibitem{kasuewashio} A. Kasue and T. Washio \emph{Growth of equivariant harmonic maps and harmonic morphisms}, Osaka J. Math. \textbf{27} (1990), 899--928. MR1088189
	
	\bibitem{korevaar} N. Korevaar, \emph{An easy proof of the interior gradient bound for solutions of the precribed mean curvature equation}, Proc. of Symp. in Pure Math., \textbf{45}, Part 2, AMS (1986). MR0843597

	\bibitem{maripessoa} L. Mari and L.F. Pessoa, \emph{Duality between Ahlfors-Liouville and Khas'minskii properties for nonlinear equations}, Comm. Anal. Geom. \textbf{28} (2020), no. 2, 395--497. MR4101343

	\bibitem{marivaltorta} L. Mari and D. Valtorta, \emph{On the equivalence of stochastic completeness, Liouville and Khas'minskii condition in linear and nonlinear setting}, Trans. Amer. Math. Soc. \textbf{365} (2013), no. 9, 4699--4727. MR3066769

	\bibitem{mir67} M. Miranda, \emph{Una maggiorazione integrale per le curvature delle ipersuperfici minimali}, Rend. Sem. Mat. Univ. Padova \textbf{38} (1967), 91--107. MR0222778
	
	\bibitem{mos61} J. Moser, \emph{On Harnack’s theorem for elliptic differential equations}, Comm. Pure Appl. Math. \textbf{14} (1961), 577--591. MR0159138 
	
	\bibitem{nr02} B. Nelli, H. Rosenberg, \emph{Minimal surfaces in $\HH^2\times\R$}, Bull. Braz. Math. Soc. (N.S.) \textbf{33} (2002), no. 2, 263--292. MR1940353
	
	\bibitem{rss13} H. Rosenberg, F. Schulze and J. Spruck, \emph{The half-space property and entire positive minimal graphs in $M\times\R$}, J. Differential Geom. \textbf{95} (2013), no. 2, 321--336. MR3128986
	
	\bibitem{sc92} L. Saloff-Coste, \emph{Uniformly elliptic operators on Riemannian manifolds}, J. Differential Geom. \textbf{36} (1992), no. 2, 417--450. MR1180389
	
	\bibitem{tru72} N. S. Trudinger, \emph{A new proof of the interior gradient bound for the minimal surface equation in $n$ dimensions}, Proc. Nat. Acad. Sci. U.S.A. \textbf{69} (1972), 821--823. MR0296832

	\bibitem{yau} S.T. Yau, \emph{Harmonic functions on complete Riemannian manifolds.} Comm. Pure Appl. Math. \textbf{28} (1975), 201--228. MR0431040
\end{thebibliography}

\end{document}